\documentclass[a4paper]{amsart}

\usepackage{amsbsy, amsmath, amsfonts, braket, latexsym, amsthm, amsxtra, amscd, graphics, mathrsfs, verbatim}
\usepackage{amssymb}
\usepackage[usenames,dvipsnames]{xcolor}

\usepackage[shortlabels]{enumitem}
\SetEnumerateShortLabel{a}{\textup{(\alph*)}}
\SetEnumerateShortLabel{A}{\textup{(\Alph*)}}
\SetEnumerateShortLabel{1}{\textup{(\arabic*)}}
\SetEnumerateShortLabel{i}{\textup{(\roman*)}}
\SetEnumerateShortLabel{I}{\textup{(\Roman*)}}

\usepackage{soul}

\usepackage[initials,lite,alphabetic]{amsrefs} 

\usepackage[all]{xy}

\usepackage{a4wide}

\theoremstyle{plain}
\newtheorem{thm}{Theorem}[section]

\newtheorem{theorem}[thm]{Theorem}

\newtheorem{lemma}[thm]{Lemma}

\newtheorem*{proper*}{Property}
\newtheorem{proposition}[thm]{Proposition}

\newtheorem{corollary}[thm]{Corollary}

\newtheorem{question}[thm]{Question}

\newtheorem*{lm*}{Lemma}
\newtheorem*{thm*}{Theorem}

\theoremstyle{definition}

\newtheorem{definition}[thm]{Definition}

\newtheorem*{df*}{Definition}

\newtheorem{example}[thm]{Example}

\newtheorem{ex-notn}[thm]{Example/Notation}

\newtheorem{observation}[thm]{Observation}

\theoremstyle{remark}

\newtheorem{remark}[thm]{Remark}

\newtheorem*{acknowledgement*}{Acknowledgement}

\newtheorem*{ex*}{Example}
\newtheorem*{exer*}{Exercise}
\newtheorem*{rem*}{Remark}
\newtheorem*{prob*}{Problem}
\newtheorem*{prop*}{Proposition}

\def\link{\operatorname{link}}

\def\Mon{\operatorname{Mon}} 

\def\Supp{\operatorname{Supp}}
\def\supp{\operatorname{supp}}

\def\frakm{\mathfrak{m}}

\def\KK{{\mathbb K}}

\def\NN{{\mathbb N}}

\def\ZZ{{\mathbb Z}}

\def\calF{\mathcal{F}}

\def\calP{\mathcal{P}}

\usepackage{bm}

\def\bda{{\bm a}}

\def\bdb{{\bm b}}

\def\bde{{\bm e}}

\def\bdg{{\bm g}}

\def\bdx{{\bm x}}

\def\deg{\operatorname{deg}}

\def\divides{{\mid}}

\def\Index#1{\emph{#1}}

\def\isom{\cong}

\def\mySet#1{\left\{ #1\right \}}

\def\th{\text{th}}

\DeclareMathOperator{\std}{std }

\DeclareMathOperator{\reg}{reg}

\def\opn#1#2{\def#1{\operatorname{#2}}} 
\opn\rev{rev}
\opn\Lex{Lex}
\opn\GL{GL}
\opn\initial{in}

\DeclareMathOperator{\suppreg}{supp.reg}

\def\AD{\mathscr{A}}

\DeclareMathOperator{\suppdeg}{supp.deg}

\title{Monomial ideals with linear quotients and componentwise (support-)linearity}
\author{Yi-Huang Shen}
\address{Wu Wen-Tsun Key Laboratory of Mathematics of CAS and School of Mathematical Sciences, University of Science and Technology of China, Hefei, Anhui, 230026, People's Republic of China}
\thanks{This work is supported by the National Natural Science Foundation of China (11201445).}
\email{yhshen@ustc.edu.cn}
\subjclass[2010]{
    05E40, 
    05E45, 
    13C14, 
    13C13. 
}
\keywords{Linear quotients; Weakly polymatroidal ideal; Weakly stable ideal; Vertex decomposability; Shellability}
\date{}

\begin{document}
\maketitle

\begin{abstract}
    When a monomial ideal has linear quotients with respect to an admissible order of increasing support-degree, we provide two proofs of different flavors to show that it is componentwise support-linear. We also introduce the variable decomposable monomial ideals. In squarefree case, these ideals correspond to the vertex decomposable simplicial complexes. We study the relationships of the variable decomposable ideals with weakly polymatroidal ideals, weakly stable ideals and ideals with linear quotients. We also investigate the componentwise properties of all these ideals.
\end{abstract}

\section{Introduction}
Throughout this paper, $S=\KK[x_1,\dots,x_n]$ is the polynomial ring in $n$ variables over the field $\KK$.  Let $I$ be a monomial ideal in $S$ with minimal monomial generating set $G(I)$.  We say that $I$ has \Index{linear quotients}, if there exists an order $\sigma=u_1,\dots,u_m$ of $G(I)$ such that each colon ideal $\braket{u_1,\dots,u_{i-1}}:u_i$ is generated by a subset of the variables for $i=2,3,\dots,m$. Any order of these generators for which we have linear quotients will be called an \Index{admissible order}.

The concept of ideals with linear quotients was introduced by Herzog and Takayama \cite{MR1918513}. Basic properties of this class of ideals were studied by Soleyman Jahan and Zheng \cite{MR2557882}. In particular, they showed that monomial ideals with linear quotients are componentwise linear.

A related concept is called \Index{componentwise support-linearity}. Through the Alexander duality introduced by Miller in \cite{MR1779598}, Sabzrou \cite{MR2561856} showed that modules with componentwise support-linearity corresponds to sequentially Cohen-Macaulay modules.

Soleyman Jahan and Zheng \cite{MR2557882} noticed that through Alexander duality, squarefree monomial ideals with linear quotients corresponds to (nonpure) shellable simplicial complexes. The Stanley-Reisner ideal of a shellable simplicial complex is always sequentially Cohen-Macaulay, cf.~\cite[page 97]{MR1453579}. This observation leads us to ask: when will a monomial ideal with linear quotients have componentwise support-linearity?

A very important property for monomial ideals with linear quotients was observed and applied extensively in Soleyman Jahan and Zheng's paper: these ideals always have a degree-increasing admissible order. Therefore, in Theorem \ref{lq-sdi} we provide a natural answer for the above question: monomial ideals with linear quotients with respect to a support-degree increasing admissible order must have componentwise support-linearity.

We provide two proofs for this result. The first approach depends on Popscu's treatment for the shellable multicomplexes in \cite{MR2338716}. This approach is easier and more direct. It provides a nice application of Miller's Alexander duality theory.

The second approach follows the strategy of \cite{MR2557882}. Therefore, we are able to manipulate the minute structures of these ideals. In particular, we can provide a slightly different proof for \cite[Corollary 2.12]{MR2557882}, which says that the facet skeletons of shellable complexes are again shellable.

The remaining part of this paper is devoted to special classes of ideals with linear quotients. Let $\Delta$ be a simplicial complex. Following \cite{MR1704158}, the Alexander dual complex $\Delta^\vee$ will be called the \Index{Eagon complex} of the Stanley-Reisner ideal $I_{\Delta}$.  Soleyman Jahan and Zheng \cite{MR2557882} observed that the Eagon complexes of squarefree monomial ideals with linear quotients are (nonpure) shellable.  Meanwhile, the Eagon complexes of squarefree weakly stable ideals \cite{MR1704158} and squarefree weakly polymatroidal ideals \cite{MR2845598} are all vertex decomposable, and vertex decomposable complexes are shellable \cite{MR1401765}.

On the other hand, on the ideal-theoretic side, even for monomial ideals which are not necessarily squarefree, we know that both weakly stable ideals and weakly polymatroidal ideals have linear quotients. Observing the correspondence and the similar implications, we have a natural question: what is the missing part that corresponds to the vertex decomposable complexes?

As the answer, we introduce the concept of \Index{variable decomposable} monomial ideals. We will show that this class fills perfectly the gap. Meanwhile, we will study the related componentwise properties of all these three types of ideals: weakly stable ideals, weakly polymatroidal ideals and our variable decomposable ideals.

\section{Preliminaries}
$\NN=\Set{0,1,2,\dots}$.
For each vector $\bda=(\bda(1),\dots,\bda(n))\in \NN^n$, we write $\bdx^\bda$ for the monomial $\prod_{i=1}^n x_i^{\bda(i)}$ in the polynomial ring $S=\KK[x_1,\dots,x_n]$. We will also write $\frakm^\bda$ for the ideal $\braket{x_i^{\bda(i)} :\bda(i)\ge 1}$. The ideals $\braket{0}$ and $\braket{1}$ will be treated as trivial monomial ideals. Following the convention, we also write $[n]$ for the set $\Set{1,2,\dots,n}$.

If $\bda\in\ZZ^n$ is a vector, the \Index{support} of $\bda$ is the set $\supp(\bda):=\Set{i\in[n] : \bda(i)\ne 0}$.  Meanwhile, if $M$ is a $\ZZ^n$-graded module and $m\in M_{\bda}$, the \Index{support} of $m$ is $\supp(m):=\supp(\bda)$ and its \Index{support-degree} is $\suppdeg(m):=|\supp(\bda)|$.

\subsection{Wedge product structure for polynomial rings}
If $i_1,i_2,\dots,i_k$ is a sequence of different integers, the pair $(r,s)$ is an \Index{inversion} with respect to this sequence precisely when $r<s$ and $i_r>i_s$.

We impose the wedge product $\wedge$ on the monomial set $\Mon(S)$ of $S$ as follows. Suppose $m_1=\bdx^{\bda}$ and $m_2=\bdx^{\bdb}$ are two monomials such that $\supp(m_1)=\Set{i_1<\cdots<i_{k}}$ and $\supp(m_2)=\Set{i_{k+1}<\cdots<i_{k+l}}$. If $\supp(m_1)\cap \supp(m_2)\ne \emptyset$, let $m_1\wedge m_2=0$. Otherwise, let $m_1\wedge m_2=(-1)^{\epsilon} m_1\cdot m_2$, where $m_1\cdot m_2$ is the usual product in $S$ and $\epsilon$ is the number of inversions in the sequence $i_1,\dots,i_{k+l}$. Now extend this binary operation $\KK$-linearly to elements in $S$.

The polynomial ring $S$ equipped with this wedge product is an associative algebra. We will write it as $(S,\wedge)$. It is generated, as a $\KK$-algebra, by elements of the form $x_i^k$ with $k\ge 1$. In particular, it is not finitely generated. Obviously, $(S,\wedge)$ contains the standard exterior algebra $\bigwedge^1(S_1)$ as a subalgebra.

Suppose that $I$ and $J$ are two monomial $S$-ideals, where the multiplicative structure is with respect to the usual product in $S$. Let $I\wedge J$ be the $\KK$-vector space spanned by
\[
    \Set{f\wedge g : f\in \Mon(I) \text{ and }g\in \Mon(J)}.
\]

\begin{lemma}
    $I\wedge J$ is the $S$-ideal generated by $f\wedge g$ with $f\in G(I)$ and $g\in G(J)$.
\end{lemma}

\begin{proof}
    Suppose $f\in \Mon(I)$ and $g\in\Mon(J)$. Then $f=f_1f_2$ with some $f_1\in G(I)$ and $g=g_1g_2$ with some $g_1\in G(I)$. If $f\wedge g\ne 0$, $\supp(f)\cap \supp(g)=\emptyset$. Hence $\supp(f_1)\cap\supp(g_1)=\emptyset$. Now $f\wedge g=\pm fg=\pm (f_2g_2)(f_1\wedge g_1)$.

    Conversely, suppose $f\in G(I)$ and $g\in G(J)$ with $\supp(f)\cap \supp(g)=\emptyset$. Take arbitrary $h\in S$, we can write $h=h_1h_2$ with $\supp(h_2)\subset \supp(g)$ and $\supp(h_1)\cap \supp(g)=\emptyset$. Now $h(f\wedge g)=\pm hfg=\pm (h_1f)(h_2g)=\pm (h_1f)\wedge (h_2g)$. Notice that $h_1f\in I$ and $h_2g\in J$.
\end{proof}

Though not directly related to the current paper, it is worth mentioning that there is a special $\KK$-linear operation $\partial$ defined on $(S,\wedge)$.  For each $f\in S$, we will have $\partial(f)=f\wedge \sigma$ where $\sigma=\sum_{k=1}^n (-1)^{k-1}x_k$. As $\sigma \wedge \sigma=0$, $\partial\circ\partial=0$. Notice that the exterior algebra $\bigwedge^1(S_1)$ with $\partial$ is closely related the classical Koszul complexes, which turned out to be very useful for commutative algebra. From $(S,\wedge)$ we can similarly construct an infinite complex that generalizes the classical construction.

\subsection{Simplicial complexes and facet ideals}

Let $\Delta$ be a simplicial complex whose vertex set is contained in $[n]$. According to \cite[Definition 2.8]{MR1333388}, for $0\le r \le s \le \dim(\Delta)$,  the \Index{skeleton} $\Delta^{(r,s)}$ is the subcomplex
\[
    \Delta^{(r,s)}=\Set{A\in \Delta : \dim(A)\le s \text{ and $A\subset F$ for some facet $F$ with $\dim(F)\ge r$}}.
\]
When $r=0$, we get the classic $s$-skeletons. When $r=s$, we get the pure $s$-skeletons.

In addition,  Soleyman Jahan and Zheng introduced the \Index{facet skeletons} of $\Delta$ as follows. To start with, the \Index{$1$-facet skeleton} of $\Delta$ is the simplicial complex
\[
    \Delta^{[1]}=\Braket{G : G\subset F \in \calF(\Delta)\text{ and }|G|=|F|-1}.
\]
Now, recursively, the \Index{$i$-facet skeleton} is defined to be the $1$-facet skeleton of $\Delta^{[i-1]}$.

For each $F\subset [n]$, we set $F^c=[n]\setminus F$. Let $\Delta^c=\braket{F^c : F\in \calF(\Delta)}$. The facet ideal of $\Delta^c$ is
\[
    I(\Delta^c)=\Braket{\bdx^{F^c} : F\in \calF(\Delta)}=\Braket{\prod_{j\notin F}x_j : F\in\calF(\Delta)}.
\]
By \cite[Lemma 1.2]{MR2083448}, the Stanley-Reisner ideal $I_{\Delta^\vee}$ of the Alexander dual complex $\Delta^\vee$ coincides with the above facet ideal:
\begin{equation}
    I_{\Delta^\vee}=I(\Delta^c).
    \label{dual-ideal}
\end{equation}
From this observation, we know that $\Delta$ is the Eagon complex of the ideal $I(\Delta^c)$.

Within this framework, we can easily deduce the following facts.

\begin{observation}
    \label{dual-observation}
    \begin{enumerate}[a]
        \item $F\in \Delta$ if and only if the squarefree monomial $\bdx^{F^c}\in I_{\Delta^\vee}$;
        \item The ideal $I_{(\Delta^{(r,s)})^\vee}$ is generated by the squarefree monomials $f\in I_{\Delta^\vee}$ such that $\deg(f)\ge n-s$ and  $f\in \braket{u}$ for some minimal monomial generator $u\in G(I_{\Delta^\vee})$ such that $\deg(u)\le n-r$.
        \item $F\in \Delta\setminus v$ if and only if $\bdx^{F^c}\in I_{\Delta^\vee}$ with $x_v\divides \bdx^{F^c}$;
        \item $F\in \link_\Delta(v)$ if and only if $x_v\divides \bdx^{F^c}$ and $\bdx^{F^c}/x_v \in I_{\Delta^\vee}$.
        \item $I_{(\Delta^{[1]})^\vee}=I_{\Delta^\vee}\wedge \frakm$.
    \end{enumerate}
\end{observation}

\section{Ideals with linear quotients}
In this section, we investigate the monomial ideals with linear quotients with respect to the componentwise support-linearity.

\begin{definition}
    Let $M$ be a finitely generated $\NN^n$-graded $S$-module.
    \begin{enumerate}[a]
        \item For each integer $d\in \Set{0,1,\dots,n}$, we denote by $M_{\braket{d}}$ the submodule of $M$ generated by all homogeneous elements $m\in M$ such that $\suppdeg(m)=d$.
        \item The \Index{support-regularity} of $M$ is
            \[
                \suppreg(M):=\max\mySet{ |\supp(\bdb)|-i : \beta_{i,\bdb}(M)\ne 0}.
            \]
        \item If $M=M_{\braket{d}}$ and $\suppreg(M)=d$ for some integer $d$, we say $M$ is \Index{$d$-support-linear}.
        \item If $M_{\braket{d}}$ is $d$-support-linear for all $d$, we say $M$ is \Index{componentwise support-linear}.
    \end{enumerate}
\end{definition}

Let $G(M)$ be a minimal homogeneous generating set for $M$.  Like the Castelnuovo-Mumford regularity, we have
\[
    \suppreg(M)\ge \max\mySet{|\supp(m)|: m\in G(M)}.
\]
We also notice that if $I$ is a monomial ideal with $I=I_{\braket{d}}$, then $I\wedge \frakm$ is precisely $I_{\braket{d+1}}$.

If $M_{\ge k}$ is the submodule of $M$ generated by all $M_j$ with $j\ge k$, the Castelnuovo-Mumford regularity of $M$ is
\[
    \reg(M)= \min\Set{k : \text{$M_{\ge k}$ is $k$-regular}}.
\]
We will analogously denote $\sum_{j\ge k}M_{\braket{j}}$ by $M_{\braket{\ge k}}$.

\begin{proposition}
    Let $M$ be a finitely generated $\NN^n$-graded $S$-module such that $\suppreg(M)=k$. Then
    \begin{enumerate}[a]
        \item $M_{\braket{k}}=M_{\braket{\ge k}}$ is $k$-support-linear, and
        \item $k=\min\Set{l\in \NN : \text{$M_{\braket{\ge l}}$ is $l$-support-linear}}$.
    \end{enumerate}
\end{proposition}

\begin{proof}
    \begin{enumerate}[a]
        \item Suppose that $G(M)=\Set{m_1,\dots,m_t}$.  Since $\suppreg(M)=k$, every $m_i$  satisfies $\suppdeg(m_i)\le k$. Thus, $M_{\braket{k}}=M_{\braket{\ge k}}$.  Suppose that $l=\min\Set{\suppdeg(m_j) | 1\le j \le t}$. Therefore, $M=M_{\braket{\ge l}}$. We assert that $\suppreg(M_{\braket{\ge s}})=k$ for each $s=l,l+1,\dots,k$. In particular, we will have $\suppreg(M_{\braket{\ge k}})=k$. Hence, $M_{\braket{k}}$ is $k$-support-linear.

            We prove the assertion by inducting on the integer $s$. The case $l=s$ is trivial. Next, we may assume that $l<s\le k$. By induction hypothesis, $\suppdeg(M_{\braket{\ge s-1}})=k$.  Notice that every nonzero homogeneous element of $M_{\braket{\ge s-1}}/M_{\braket{\ge s}}$ has support-degree $s-1$. Thus, $\suppreg(M_{\braket{\ge s-1}}/M_{\braket{\ge s}})=s-1$ by \cite[Lemma 2.10(2)]{MR2561856}.  Now, the assertion for $s$ follows from the short exact sequence
            \[
                0\to M_{\braket{\ge s}} \to M_{\braket{\ge s-1}} \to M_{\braket{\ge s-1}}/M_{\braket{\ge s}} \to 0
            \]
            and \cite[Lemma 2.7]{MR2561856}.
        \item Let $m=\min\Set{l\in [n] : \text{$M_{\braket{\ge l}}$ is $l$-support-linear}}$. As $M_{\braket{\ge k}}$ is $k$-support-linear, we have $m\le k$. If $m<k$, the proof for part (a) implies that $\suppreg(M_{\braket{\ge m}})=k>m$. Thus $M_{\braket{\ge m}}$ is not $m$-support-linear, which is absurd.  \qedhere
    \end{enumerate}
\end{proof}

\begin{definition}
\begin{enumerate}[a]
\item We say that the monomial ideal $I$ has \Index{linear quotients}, if there exists an order $\sigma=u_1,\dots,u_m$ of $G(I)$ such that each colon ideal $\braket{u_1,\dots,u_{i-1}}:u_i$ is generated by a subset of the variables for $i=2,3,\dots,m$. Any order of these generators for which we have linear quotients will be called an \Index{admissible order}. If
    \[
        \deg(u_1)\le \deg(u_2) \le \cdots \le \deg(u_m),
    \]
    this order is \Index{degree increasing}.  Likewise, if
    \[
        \suppdeg(u_1)\le \suppdeg(u_2) \le \cdots \le \suppdeg(u_m),
    \]
    this order is \Index{support-degree increasing}.
            \item The monomial ideal $I$ has \Index{componentwise linear quotients} (resp.~\Index{support-componentwise linear quotients}) if for each $d$, the ideal $I_{d}$ (resp.~$I_{\braket{d}}$) has linear quotients.
    \end{enumerate}
\end{definition}

Obviously, $\sigma=u_1,\dots,u_m$ is an admissible order of $I$  if and only for each pair $i<j$, we can find $k<j$ and $d\in[n]$ such that $\braket{u_k}:u_j=\braket{x_d}\supset \braket{u_i}:u_j$. We will use repeatedly the fact that $\braket{u_i}:u_j=\braket{\frac{u_i}{\gcd(u_i,u_j)}}$.

We will treat principal monomial ideals as trivial ideals with linear quotients. Notice that monomial ideals with linear quotients always have degree-increasing admissible orders, by \cite[Lemma 2.1]{MR2557882}.

It is well known (cf.~\cite[Lemma 4.1]{MR1995137}) that if $I$ has linear quotients, then 
\[
\reg(I)=\max\Set{\deg(u) : u\in G(I)}.
\]
 We have a support degree version for this result.

\begin{proposition}
    Suppose that the monomial ideal $I$ has linear quotients with respect to the admissible order $u_1,\dots,u_m$. Then $\suppreg(I)=\max\Set{\suppdeg(u_i) : i=1,2,\dots,m}$.
\end{proposition}

\begin{proof}
    Suppose that the $\max\Set{\suppdeg(u_i) : i=1,2,\dots,m}=d$.

    For each $k=1,2,\dots,m$, let $I_k=\braket{u_1,\dots,u_k}$ and $L_k=I_{k-1}:u_k$ with $I_0=0$. We show by induction that $\suppreg(I_k)\le d$.  This claim is clear when $k=1$. Thus we may suppose that $k\ge 2$. Suppose that $u_k=\bdx^{\bdb}$. We have the following short exact sequence of $\NN^n$-graded modules
    \[
        0\to I_{k-1} \to I_k \to I_k/I_{k-1} \to 0.
    \]
    But $I_k/I_{k-1}\cong (S/L_k)(-\bdb)$ and $\suppreg((S/L_k)(-\bdb))=|\supp(\bdb)|\le d$. As $\suppreg(I_{k-1})\le d$ by induction hypothesis, we can conclude from \cite[Lemma 2.7]{MR2561856} that $\suppreg(I_k)\le d$. Thus, $\suppreg(I)\le d$.

    On the other hand, there is one generator of $I$ with support-degree $d$. Therefore, $\suppreg(I)\ge d$.
\end{proof}

\begin{corollary}
    \label{lq-d}
    Suppose that the monomial ideal $I$ has linear quotients with respect to the admissible order $u_1,\dots,u_m$ and $\suppdeg(u_i)=d$ for each $i$. Then $I$ is $d$-support-linear.
\end{corollary}

An \Index{irreducible} monomial ideal is a monomial ideal that cannot be written as proper intersection of two other monomial ideals. It follows from \cite[Corollary 1.3.2]{MR2724673} that irreducible monomial ideals are precisely those ideals which are generated by pure powers of the variables.

\begin{definition}
    We say that the monomial ideal $I$ has \Index{Popescu quotients}, if there exists an order $\sigma=u_1,\dots,u_m$ of $G(I)$ and index $s$ with $1\le s \le m$, such that
    \begin{enumerate}[a]
        \item $\supp(u_1)=\supp(u_j)$ for all $1\le j \le s$;
        \item\label{b} for each $i=s+1,s+2,\dots,m$, the colon ideal
            $\braket{u_1,\dots,u_{i-1}}:u_i$ is irreducible;
        \item whenever $\supp(u_i)\subset \supp(u_j)$, then
            $\supp(u_i)=\supp(u_j)$ or $i<j$.
    \end{enumerate}
    Any order of these generators for which we have Popescu quotients will be called an \Index{admissible Popescu order}. If the sequence only satisfies the condition \ref{b} with $s=1$, we say that $I$ has \Index{weak Popescu quotients} and the order is a \Index{weak admissible Popescu order}.
\end{definition}

Obviously, monomial ideals which have linear quotients with respect to support-degree increasing admissible orders, also have Popescu quotients.  And the squarefree monomial ideals that have weak Popescu quotients are exactly those squarefree monomial ideals that have linear quotients.

When $\bda,\bdg\in\NN^n$ with $\bda\le\bdg$, let $\bdg\setminus \bda$ be the vector whose $i^{\text{th}}$ coordinate is
\[
    \bdg(i)\setminus \bda(i):=
    \begin{cases}
        \bdg(i)+1-\bda(i), & \text{if $\bda(i)\ge 1$},\\
        0, & \text{if $\bda(i)=0$}.
    \end{cases}
\]
When  $I$ is a monomial ideal whose minimal monomial generators all
divide $\bdx^\bdg$, the \Index{Alexander dual ideal} of $I$ with respect to
$\bdg$ is
\[
    I^{[\bdg]}:=\bigcap \Set{\frakm^{\bdg\setminus\bda} : \bdx^\bda\in G(I)}.
\]
In \cite{MR2110098}, Miller defined an Alexander duality functor $\AD_\bdg$ for $\bdg$-determined modules. In our case, $I$ is $\bdg$-determined and $\AD_\bdg(I)$ is precisely the quotient ring $S/I^{[\bdg]}$. Notice that $S/I^{[\bdg]}$ is also $\bdg$-determined and the Alexander duality functor satisfies $\AD_\bdg(\AD_\bdg(I))=I$.

For each generator $u_i=\bdx^{\bda_i}\in G(I)$, let $\bda_i^\ast\in \NN_\infty^n$ be the vector whose $k^{\th}$ coordinate is
\[
    \bda_i^\ast(k)=
    \begin{cases}
        \bdg(k)-\bda_i(k),& \text{if $\bda_i(k)>1$},\\
        +\infty,& \text{if $\bda_i(k)=0$}.
    \end{cases}
\]
Let $\Gamma_I=\Gamma(\bda_1^*,\dots,\bda_s^*)$ be the multicomplexes generated by $\bda_1^\ast,\dots,\bda_s^\ast$ in the sense of \cite[Definition 9.2]{MR2267659}.  We write $I(\Gamma_I)$ for the ideal of nonfaces in $\Gamma_I$. By \cite[Proposition 9.12]{MR2267659}, we have
\[
    I(\Gamma(\bda_1^*,\dots,\bda_s^*))=\bigcap_{j=1}^s I(\Gamma(\bda_j^*)) =\bigcap_{j=1}^s \frakm^{\bdg\setminus \bda_i} =I^{[\bdg]}.
\]
Thus, $S/I(\Gamma_I)=\AD_\bdg(I)$.

\begin{observation}
    The ideal $I$ has Popescu quotients with respect to the sequence $u_1,\dots,u_s$ if and only if $\Gamma$ is maximal shellable in the sense of \cite{MR2338716} with respect to the sequence of maximal facets $\bda_1^\ast,\dots,\bda_s^\ast$.
\end{observation}

\begin{theorem}
\label{lq-sdi}
    If the monomial ideal $I$ has linear quotients with respect to some support-degree increasing admissible order, then it is componentwise support-linear.
\end{theorem}

\begin{proof}
    As the ideal $I$ has Popescu quotients, the corresponding multicomplex $\Gamma_I$ constructed above is maximal shellable and hence shellable by \cite[Theorem 3.6]{MR2338716}. Consequently, the quotient ring $S/I(\Gamma)$ is sequentially Cohen-Macaulay by  \cite[Corollary 10.6]{MR2267659}. But by the correspondence established by \cite[Theorem 2.11]{MR2561856}, this is equivalent to saying that the ideal $I=\AD_\bdg(S/I(\Gamma_I))$ is componentwise support-linear.
\end{proof}

In Corollary \ref{cpt-sl}, we will provide another proof for the previous result.

\begin{proposition}
    [{cf.~\cite[Lemma 2.5]{MR2557882}}]
    \label{Im}
    If monomial ideal $I\subset S$ has linear quotients with respect to a support-degree increasing admissible order,  then the ideal $I\wedge \frakm$ also has linear quotients with respect to a support-degree increasing admissible order.
\end{proposition}

\begin{proof}
    We may suppose that $I$ has linear quotients with respect to a support-degree increasing admissible order $\sigma=u_1,\dots,u_m$.  Let
    \[
        T=\Set{(i,j) | 1\le i \le m \text{ and }1\le j \le n}.
    \]
    We equip $T$ with a linear order:
    \[
        (i_1,j_1)\prec (i_2,j_2) \text{ if and only if }i_1<i_2 \text{ or }i_1=i_2 \text{ with }j_1<j_2.
    \]
    Let $\phi: T \to S$ defined by $(i,j)\mapsto u_ix_j$.  The ideal $I\wedge \frakm$ is generated by $\phi(\widetilde{T})$ where
    \[
        \widetilde{T}=\Set{ (i,j)\in T : j\notin \supp(u_i)}.
    \]
    We remove redundant elements in $\widetilde{T}$ following this rule: if $\phi(i,j)$ and $\phi(r,s)$ are two elements with $i<r$ and $\phi(i,j)\divides \phi(r,s)$, we remove $(r,s)$. Now, we end up with a subset $\widetilde{T}'\subset \widetilde{T}$.  We equip these two subsets of $T$ with the inherited linear order.  Obviously, $I\wedge \frakm=\braket{\phi(\widetilde{T}')}$.

    We will show that $I\wedge \frakm$ has linear quotients with respect to the monomials in $\widetilde{T}'$ in the given order which is clearly support-degree increasing.  The case when $m=1$ is clear. Thus, we may assume $m>1$ and by induction assume that $\braket{u_1,\dots,u_{m-1}}\wedge \frakm$ has linear quotients with respect to the linearly-ordered subset $\Set{(i,j)\in \widetilde{T}' : 1\le i\le m-1}$. This subset forms an initial piece of $\widetilde{T}'$.

    Take a pair $(m,j)\in \widetilde{T}'$ and assume that $J=\braket{\phi(r,s) : (r,s)\prec (m,j)\in \widetilde{T}'}$. We show that $J:u_mx_j$ is generated by some monomials of degree $1$.

    Let $(k,l)\prec (m,j)\in \widetilde{T}'$. If $k=m$, obviously $\braket{u_kx_l}:u_mx_j=\braket{x_l}$. Now suppose that $k<m$. Thus, there is some $q<m$ such that
    \[
        \braket{u_q}:u_m=\braket{x_t}\supset \braket{u_k}:u_m.
    \]

    We claim that $x_t\ne x_j$. Otherwise, since $\braket{u_q}:u_m=\braket{x_t}$, we have $u_qw=x_ju_m$ for some monomial $w\in S$. As $j\notin \supp(w)$, we have $j\in\supp(u_q)$ and $u_q=x_ju$ for some monomial $u$. Now $1\ne w=\frac{u_m}{u_q/x_j}$ and $\deg_{x_j}(u_q)=1$. Since $\suppdeg(u_m)> \suppdeg(u_q/x_j)=\supp(u_q)-1$, we can find some $d\ne t$ such that $d\in\supp(w)\setminus \supp(u_q)$. But then $x_du_q=x_dux_j\divides wux_j=u_mx_j$. This implies that $(m,j)\notin \widetilde{T}'$, which is absurd. Hence $x_t\ne x_j$.

    We claim that $j\notin \supp(u_q)$.  Otherwise, since $\braket{u_q}:u_m=\braket{x_t}$, we have $u_qf=u_mx_t$ for some monomial $f\in S$ with $t\notin \supp(f)$. The assumption that $j\in \supp(u_q)$ implies $x_j | u_mx_t$. As $j\ne t$, we have $x_j\divides u_m$, contradicting our choice of $x_j$.  Hence $j\notin \supp(u_q)$.

    Now $j\not\in \supp(u_q)$ and $(q,j)\in \widetilde{T}$.  Correspondingly there is some $(r,s)\in \widetilde{T}'$  preceding $(m,j)$ such that $u_rx_s \divides u_qx_j$. We claim that
    \begin{equation}
        \braket{u_rx_s}:u_mx_j=\braket{x_t} \supset \braket{u_kx_l}:u_mx_j.
        \label{clm-1}
    \end{equation}

    As $\braket{u_rx_s}:u_mx_j$ is principal and $\braket{u_rx_s}:u_mx_j\supset \braket{u_qx_j}:u_mx_j=\braket{x_t}$, if $\braket{u_rx_s}:u_mx_j\ne \braket{x_t}$, we must have $\braket{u_rx_s}:u_mx_j=S$. But this is equivalent to saying that $u_rx_s \divides u_mx_j$ and $(m,j)\notin \widetilde{T}'$, which cannot happen. Thus $\braket{u_rx_s}:u_mx_j=\braket{x_t}$.

    Since $\braket{x_t}\supset \braket{u_k}:u_m$, we conclude that $t\in \supp(u_k/\gcd(u_k,u_m))$. As $x_j\ne x_t$, this implies that 
    \[
        t\in \supp(u_kx_l/\gcd(u_kx_l,u_mx_j)).
    \]
    Hence $\braket{x_t} \supset \braket{u_kx_l}:u_mx_j$.

    Since we have established that $J:u_mx_j$ is generated by monomials of degree $1$, the induction process shows that $I\wedge \frakm$ has linear quotients with respect to the given support-degree increasing admissible order.
\end{proof}

\begin{corollary}
    \label{LQmI}
    Let $I\subset S$ be a monomial ideal with linear quotients. If all its minimal monomial generators have support-degree $d$, then $I_{\braket{d+1}}$ has linear quotients.
\end{corollary}

\begin{corollary}
    [{\cite[Corollary 2.12]{MR2557882}}]
    \label{facet-skeletons}
    The facet skeletons of shellable simplicial complexes are again shellable.
\end{corollary}

\begin{proof}
    Soleyman Jahan and Zheng \cite{MR2557882} observed that a simplicial complex is shellable if and only it is the Eagon complex  of some squarefree ideal with linear quotients. They also showed that any monomial ideal with linear quotients has a degree-increasing admissible order.  
    
    Now, assume that the shellable simplicial complex $\Delta$ is the Eagon complex of the monomial ideal $I$. Since $I$ is squarefree, this degree-increasing admissible order is automatically support-degree increasing.  Thus, by Proposition \ref{Im}, $I\wedge \frakm$ has linear quotients.  Notice that $\Delta^{[1]}$ is the Eagon complex of $I\wedge \frakm$.  Thus, $\Delta^{[1]}$ is also shellable.
\end{proof}

\begin{theorem}
    [{cf.~\cite[Theorem 2.7]{MR2557882}}]
    \label{cpt-lq}
    If monomial ideal $I$ has linear quotients with respect to some support-degree increasing admissible order, then $I_{\braket{\ge d}}$ has linear quotients for all $d\in [n]$.
\end{theorem}

\begin{proof}
    It suffices to consider the special case when $I$ has linear quotients with respect to a support-degree increasing admissible order $u_1,\dots,u_s,v_1,\dots,v_t$, where $\suppdeg(u_i)=a$ for all $i$ and $\suppdeg(v_j)\ge a+1$ for all $j$. Now we show that $I_{\braket{\ge a+1}}=\braket{u_1,\dots,u_s}\wedge \frakm + \braket{v_1,\dots,v_t}$ has linear quotients with respect to some support-degree increasing admissible order.

    Let $w_1,\dots,w_l$ be the support-degree increasing admissible order for $\braket{u_1,\dots,u_s}\wedge \frakm$, as constructed for Proposition \ref{Im}. Then $I_{\braket{\ge a+1}}$ is minimally generated by $w_1,\dots,w_l,v_1,\dots,v_t$. We only need to show that $\braket{w_1,\dots,w_l,v_1,\dots,v_{p-1}}:v_p$ is generated by some monomials of degree $1$ for $1\le p\le t$.

    We have two cases. First, we consider $\braket{v_j}:v_p$ with $j<p$.  Since $u_1,\dots,u_s,v_1,\dots,v_p$ is an admissible order, there is some $u\in \Set{u_1,\dots,u_s,v_1,\dots,v_{p-1}}$ and some $d\in[n]$ such that $\braket{u}:v_p=\braket{x_d}\supset \braket{v_j}:v_p$. If $u\in \Set{v_1,\dots,v_{p-1}}$, we are done.  Thus, we may assume that $u\in \Set{u_1,\dots,u_s}$. As $\braket{u}:v_p=\braket{x_d}$, $x_dv_p=uf$ for some monomial $f\in S$.  Obviously $d\notin \supp(f)$. Hence $d\in \supp(u)$ and $u/x_d\in S$.  Now $f=\frac{v_p}{u/x_d}$. As $\suppdeg(v_p)>\supp(u)\ge \supp(u/x_d)$, we can find some $c\in \supp(v_p)\setminus \supp(u)$.  As $d\in \supp(u)$, we have $d\ne c$. Now, by the choice of $w_1,\dots,w_l$, we have some $w_k$ that divides $ux_c$. We claim that $\braket{w_k}:v_p=\braket{x_d}\supset \braket{v_j}:v_p$. The argument for this claim is analogous to that for \eqref{clm-1} in the proof for Proposition \ref{Im}. There is no need to repeat here.

    Next, we consider $\braket{w_j}:v_p$. This $w_j$ equals $u_ix_j$ for suitable $i$ and $j$. As $u_1,\dots,u_s,v_1,\dots,v_t$ is an admissible order, there is some $u\in \Set{u_1,\dots,u_s,v_1,\dots,v_{p-1}}$ and some $d\in [n]$ such that $\braket{u}:v_p=\braket{x_d}\supset \braket{u_i}:v_p$.  Notice that $\braket{u_i}:v_p \supset \braket{w_j=u_ix_j}:v_p$. Now, if $u\in \Set{v_1,\dots,v_p}$, we are done. Hence, we may assume that $u\in \Set{u_1,\dots,u_s}$.  Again, we are able to find some $c\in \supp(v_p)\setminus \supp(u)$ with $d\ne c$ and some $w_k$ that divides $ux_c$. Analogously, we will have $\braket{w_k}:v_p=\braket{x_d}\supset \braket{w_j}:v_p$. And this completes the proof.
\end{proof}

Support-linear modules are componentwise support-linear by \cite[Lemma 2.1]{MR2561856}.  We have a similar result for linear quotients property.

\begin{corollary}
    [{cf.~\cite[Corollary 2.8]{MR2557882}}]
    \label{cpt-sl}
    Let $I\subset S$ be a monomial ideal. If $I$ has linear quotients with respect to some support-degree increasing admissible order, then $I$ has support-componentwise linear quotients and componentwise support-linearity.
\end{corollary}

\begin{proof}
    The first part follows from the proof for Theorem \ref{cpt-lq}.  The second part follows from the first part together with  Corollary
    \ref{lq-d}.
\end{proof}

\begin{corollary}
    [{\cite[Theorem 2.9]{MR1333388}}]
    \label{skeletons}
    If $\Delta$ is shellable, then all its skeletons $\Delta^{(r,s)}$ are shellable as well. .
\end{corollary}

The proof is similar to that for Corollary \ref{facet-skeletons}.

\begin{remark}
    We observe that
    \begin{enumerate}[a]
        \item ideals with linear quotients in general do not have any support-degree increasing admissible order, and
        \item ideals with linear quotients in general are not componentwise support-linear.
    \end{enumerate}
    For instance, the ideal $I=\braket{a^2b,abc,bcd,cd^2}$ has linear quotients. But $I_{\braket{2}}=\braket{a^2b,cd^2}$ does not have linear quotients.  Furthermore, $I_{\braket{2}}$ is not $2$-support-linear. Thus, the requirement that \textit{$I$ has a support-degree increasing admissible order} in our results is essential.
\end{remark}

\begin{remark}
    When $I$ is a monomial ideal with linear quotients and support-componentwise linear quotients, it might still have no support-degree increasing admissible order. The ideal
    \[
        I=\braket{bc,abd^2,b^3d^2,cd,ac,c2,a^2bd}
    \]
    provides such an example. It has linear quotients with respect to the given order. All of the components $I_{\braket{1}}=\braket{c^2}$, $I_{\braket{2}}=\braket{bc,cd,ac,b^3d^2}$ and $I_{\braket{3}}=\braket{abc,acd,bcd,abd^2,a^2bd}$ also have linear quotients.  On the other hand, $I$ has no support-degree increasing admissible order.  It is worth noting that in any such an example, the ideal $I$ is not squarefree by \cite[Lemma 2.1]{MR2557882}.
\end{remark}

\begin{proposition}
    [{cf.~\cite[Proposition 2.9]{MR2557882}}]
    \label{Pack}
    Let $I$ be a monomial ideal with support-componentwise linear quotients. Assume that for each component $I_{\braket{d}}$ there exists an admissible order $\sigma_d$ of $G(I_{\braket{d}})$ with the property that the elements of $G(I_{\braket{d}}\wedge \frakm)$ form the initial part of $\sigma_{d+1}$. Then $I$ has linear quotients with respect to a support-degree increasing admissible order.
\end{proposition}

\begin{proof}
    The monomials $u_1,\dots,u_s$ in $G(I)$ can be ordered such that $i<j$ if $\suppdeg(u_i)<\suppdeg(u_j)$ or $\suppdeg(u_i)=\suppdeg(u_j)=d$ and $u_i$ comes before $u_j$ in $\sigma_d$. This order is clearly support-degree increasing.  We claim that $I$ has linear quotients with respect to this order.

    It suffices to show that the colon ideal $\braket{u_1,\dots,u_{p-1}}:u_p$ is generated by variables. We may assume that $p>1$ and $\suppdeg(u_1)<\suppdeg(u_p)=d$. Let $l<p$ be the largest index such that $\suppdeg(u_l)<d$. Then we have an admissible order $w_1,\dots,w_t,u_{l+1},\dots,u_p$ where $w_1,\dots,w_t\in G(I_{\braket{d-1}}\wedge \frakm)$.

    Let $j<p$. Since $\suppdeg(u_j)\le \suppdeg(u_p)$, we can find suitable monomial $m$ such that $\deg(m)=\suppdeg(m)=\suppdeg(u_p)-\suppdeg(u_j)$ and $\supp(m)\subset \supp(u_p)\setminus \supp(u_j)$. Therefore, $\suppdeg(mu_j)=\suppdeg(u_p)$ and $\braket{mu_j}:u_p=\braket{u_j}:u_p$.
    This $m$ is a product of distinct variables. Hence $mu_j\in \Set{w_1,\dots,w_t,u_{l+1},\dots,u_{p-1}}$. Now, we can find suitable $w\in\Set{w_1,\dots,w_t,u_{l+1},\dots,u_{p-1}}$ and $d\in [n]$ such that $\braket{w}:u_p=\braket{x_d}\supset \braket{mu_j}:u_p=\braket{u_j}:u_p$.

    There is nothing to show when $w\in \Set{u_{l+1},\dots,u_{p-1}}$. Thus, we may assume that $w\in \Set{w_1,\dots,w_t}$. In this case, $w=m'u_i$ for some $i\le l$ and some monomial $m'$. As  $\braket{u_i}:u_p\ne S$ is a principal monomial ideal containing $\braket{m'u_i}:u_p=\braket{x_d}$, we must have $\braket{u_i}:u_p=\braket{x_d}\supset \braket{x_j}:u_p$.  This completes the proof.
\end{proof}

\begin{example}
The compatibility requirement in Proposition \ref{Pack} is essential.
Let 
\[
I=\braket{x_2^4,x_1x_2^3,x_2^3x_3,x_1^2x_2x_3}\subset \KK[x_1,x_2,x_3].
\]
 Then $I$ has support-componentwise linear quotients. But $I$ does not have linear quotients. 
\end{example}

The following question was implicitly asked in \cite{MR2557882} and remains open so far.

\begin{question}
Do monomial ideals with componentwise linear quotients have linear quotients? 
\end{question}

\section{Skeletons and other classes of monomial ideals}
In this section, we are concerned with the following types of questions:

\begin{question}
    \label{Skeleton}
    \begin{enumerate}[a]
        \item If $I$ is a monomial ideal with property $\calP$, do $\frakm I$ and $\frakm \wedge I$ also have this property?
        \item\label{ques-b} If $\Delta$ is a simplicial complex with property $\calP$, do all the skeletons of $\Delta$ also have this property?
    \end{enumerate}
\end{question}

We will study these questions with respect to the weakly $I$-stable ideals, weakly polymatroidal ideals and vertex decomposable complexes. They are all related to the ideals with linear quotients and shellable complexes that we investigated in the previous section.

\subsection{$I$-stable ideals}

Fix an irreducible monomial ideal $I$. For each monomial ideal $u\in S$, let $\max(u)=\max(\Supp(u))$ with $\max(1)=-\infty$.
 When $u\ne 1$, we also set $u'=u/x_{\max(u)}$.

\begin{definition}
    [{cf.~\cite[Definitions 2.1--2.3]{MR1921814}}]
    \label{IS} 
    Let $J$ be a monomial ideal in $S$ with $G(J)\cap I=\emptyset$.
    \begin{enumerate}[a]
        \item $J$ is said to be \Index{$I$-stable} if the following condition holds for each monomial $u\in J\setminus I$:
            \begin{itemize}
                \item[(IS):] for each $j<\max(u)$, there exists $i\in \supp(u)$ with $i>j$ and $ux_j/x_i\in J+I$.
            \end{itemize}
        \item $J$ is said to be \Index{weakly $I$-stable} if the following condition holds for each monomial $u\in J\setminus I$:
            \begin{itemize}
                \item[(WIS):] for each $j<\max(u')$, there exists $i\in \supp(u)$ with $i>j$ and $ux_j/x_i\in J+I$.
            \end{itemize}
        \item $J$ is said to be \Index{strongly $I$-stable} if the following condition holds for each monomial $u\in J\setminus I$:
            \begin{itemize}
                \item[(SIS):] $ux_j/x_i\in J+I$ for every $i\in \supp(u)$ and $j<i$.
            \end{itemize}
    \end{enumerate}
\end{definition}

Definitions above are generalizations of those in \cite{MR1037391}, \cite{MR1618728} and \cite{MR1218500}.  Although these definitions work for general monomial ideal $I$, the mostly interesting cases happen when $I$ is irreducible, cf.~\cite{MR1921814}.  Notice that \Index{weakly stable} property was previously introduced for squarefree monomial ideals only. This corresponds to the weakly $I$-stable case when $I=\braket{x_1^2,\dots,x_n^2}$. Similarly, a squarefree monomial ideal will be called \Index{squarefree stable} (resp.~\Index{squarefree strongly stable}) if it is $I$-stable (resp.~strongly $I$-stable) for this choice of $I$. 
The implications
\[
    \text{strongly $I$-stable} \implies \text{$I$-stable} \implies \text{weakly $I$-stable}
\]
are obvious. 

For every monomial ideal $J$, we call $\std_I(J):=\braket{u : u\in G(J)\setminus I}$ the \Index{standard form} of $J$ with respect to $I$. Obviously it is the unique minimal monomial ideal $K$ such that $K+I=J+I$.

\begin{lemma}
    [{\cite[Lemma 2.7]{MR1921814}}]
    Let $J=\std_I(J)$ be a monomial ideal. Then $J$ is $I$-stable (reps.~weakly $I$ stable, strongly $I$-stable) if and only if each $u\in G(J)$ satisfies the condition (IS) (reps.~(WIS), (SIS)) in Definition \ref{IS}.
\end{lemma}

We will use the following term order $\prec$ on $\Mon(S)$ throughout this subsection: $\bdx^\bda \prec \bdx^\bdb$ if and only if $\deg(\bdx^\bda)<\deg(\bdx^\bdb)$ or $\deg(\bdx^\bda)=\deg(\bdx^\bdb)$ and there exists some $s\in [n]$ such that $\bda(k)=\bdb(k)$ for all $s+1\le k \le n$, but $\bda(s)<\bdb(s)$. This order is closely related to the $I$-stable ideals.

\begin{lemma}
    \label{WIS-deg}
    Suppose that $J=\std_I(J)$ is a monomial ideal and $G(J)=\Set{u_1,\dots,u_s}$ with $u_1\prec u_2\prec \cdots \prec u_s$.  If $J$ is $I$-stable (resp.~weakly $I$ stable, strongly $I$-stable), then for each $k\in [s]$, the ideal $J_k=\braket{u_1,u_2,\dots,u_k}$ is also $I$-stable (resp.~weakly $I$ stable, strongly $I$-stable).
\end{lemma}

For the weakly $I$-stable case, it is \cite[Lemma 7.1]{MR1921814}. Its proof also works for the $I$-stable and strongly $I$-stable cases.

\begin{proposition}
    \label{WIS}
    Suppose that $J$ is a monomial ideal. If $J$ is $I$-stable (resp.~weakly $I$ stable, strongly $I$-stable), then $\std_I(\frakm J)$ is also $I$-stable (reps.~weakly $I$ stable, strongly $I$-stable).
\end{proposition}

\begin{proof}
    We will only consider the $I$-stable case. The other two cases are similar.

    Suppose that  $u_1\prec u_2 \prec \cdots \prec u_s$ are the monomials in $G(J)$.  Now, for each $u\in \frakm J \setminus I$, let $k$ be the smallest index such that $u\in J_k=\braket{u_1,u_2,\dots,u_k}$. Thus, $u\in \frakm J_k\setminus I$ and $\deg(u)>\deg(u_k)$. By Lemma \ref{WIS-deg}, $J_k$ is again weakly $I$-stable. Thus, for each $j<\max(u)$, there exists $i\in \supp(u)$ with $i>j$ and $ux_j/x_i\in J_k+I$. If $ux_j/x_i\in I$, we are done. If $ux_j/x_i\in J_k$, as $\deg(ux_j/x_i)=\deg(u)>\deg(u_k)\ge \deg(u_i)$ for all $u_i\in G(J_k)$, we have $ux_j/x_i\in \frakm J_k\subset \frakm J$.
\end{proof}

\begin{corollary}
    \label{Co}
    If $J$ is squarefree stable (resp.~squarefree weakly stable, squarefree strongly stable), then $J\wedge \frakm$ is also squarefree stable (resp.~squarefree weakly stable, squarefree strongly stable).
\end{corollary}


\begin{proposition}
    \label{WIS-LQ}
    Weakly $I$-stable ideals have linear quotients.
\end{proposition}

It follows directly from Lemma \ref{WIS-deg} and the following

\begin{lemma}
    [{\cite[Lemma 7.2]{MR1921814}}]
    Suppose that $J=\std_I(J)$ is a monomial ideal.  Let $v\in S\setminus (I+J)$ be a monomial with $\deg(v)=a$ such that $J'=\braket{J,v}$ is a weakly $I$-stable ideal. If $deg(u)\le a$ for every $u\in G(J)$, then $J'/J\isom S/L(-a)$ where $L=\braket{x_t: vx_t\in J\setminus I}$.
\end{lemma}

\begin{remark}
    \label{cpt}
    Let $J=\std_I(J)$ be a monomial ideal.
    One might call $J$ \Index{componentwise $I$-stable} if for each degree $d$, the component ideal $\braket{\std_I(J_d)}$ is $I$-stable. But it follows easily from definition that $J$ is $I$-stable if and only if $J$ is componentwise $I$-stable. The same is true for weakly $I$-stable and strongly $I$-stable properties.
\end{remark}

\begin{definition}
    A simplicial complex $\Delta$ is \Index{co-stable} (resp.~\Index{weakly co-stable}, \Index{strongly co-stable}) if it is the Eagon complex of a squarefree stable (resp.~squarefree weakly stable, squarefree strongly stable) ideal.
\end{definition}

We have the following implications:
\[
    \text{strongly co-stable} \implies \text{co-stable} \implies \text{weakly co-stable}.
\]

Now, having Corollaries \ref{facet-skeletons}, \ref{skeletons} and Remark \ref{cpt} in mind, we provide the following answer regarding Question \ref{Skeleton}\ref{ques-b} in the current framework:

\begin{corollary}
    If $\Delta$ is co-stable (resp.~weakly co-stable, strongly co-stable), then all its facet skeletons $\Delta^{[i]}$ and skeletons $\Delta^{(r,s)}$ are also co-stable (resp.~weakly co-stable, strongly co-stable).
\end{corollary}

\begin{proof}
    We apply Observation \ref{dual-observation} to translate combinatorial properties of simplicial complexes to properties of squarefree monomial ideals.
    For the facet skeletons part, we simply apply Corollary \ref{Co}. For the skeletons $\Delta^{(r,s)}$ part, we observe as in Remark \ref{cpt}.
\end{proof}

\subsection{Weakly polymatroidal ideals}

Let $u=\bdx^{\bda}$ and $v=\bdx^{\bdb}$ be two distinct monomials in $S$. When $u\succ_{lex} v$ lexicographically, there exists an index $t$ such that $\bda(i)=\bdb(i)$ for $1\le i<t$ and $\bda(t)>\bdb(t)$. This is a term order for $\Mon(S)$.

\begin{definition}
    [{cf.~\cite[Definition 1.1]{MR2724673}}]
    \begin{enumerate}[a]
        \item A monomial ideal $I$ is called \Index{weakly polymatroidal} if for every two monomials $u\succ_{lex} v\in G(I)$ the following condition is satisfied:
            \begin{itemize}
                \item[(WP):]   if $t$ is the smallest index such that $\deg_{x_t}(u)>\deg_{x_t}(v)$, then there exists $j>t$ such that $x_t(v/x_j)\in I$.
            \end{itemize}
        \item The monomial ideal $I$ is called \Index{componentwise weakly polymatroidal} (resp.~\Index{support-componentwise weakly polymatroidal}) if for each $d$, the ideal $I_{d}$ (resp.~$I_{\braket{d}}$) is weakly polymatroidal.
    \end{enumerate}
\end{definition}

We will treat principal monomial ideals as trivial weakly polymatroidal ideals. Notice that the original definition of weakly polymatroidal property can be traced back to \cite{MR2260118} and requires the minimal monomial generators of $I$ to be in one degree.

Evidently, the weakly polymatroidal property is closely related to the lexicographic order $\succ_{lex}$ of the monomials. Notice that $x_t(v/x_j)\succ_{lex} v$ in the above definition, although we don't necessarily have $u\succeq_{lex}x_t(v/x_j)$.

Mohammadi and Moradi \cite[Theorem 1.6]{MR2768496} proved that if $I$ is weakly polymatroidal, then $I\frakm$ is again weakly polymatroidal. If additionally $I$ is generated by monomials in one degree, then $I$ is componentwise weakly polymatroidal, by \cite[Corollary 1.7]{MR2768496}. Unlike the linear quotients case, the degree requirement cannot be removed, as shown by \cite[Example 1.8]{MR2768496}.

\begin{theorem}
    \label{WPP}
    If ideal $I$ is a weakly polymatroidal ideal generated by monomials in one support-degree, then $I\wedge \frakm$ is again weakly polymatroidal.
\end{theorem}

\begin{proof}
    Take two different elements $w_1\succ_{lex}w_2$ in $G(I\wedge \frakm)$.  Let $u\in G(I)$ be the greatest with respect to lexicographical order such that $w_1=x_iu$ for some $i\in [n]\setminus \supp(u)$. This implies that if $w_1=x_{i'}u'$ for another $u'\in G(I)$ and $i'\in[n]\setminus \supp(u')$, then $i>i'$. Similarly, we choose $v\in I$ for $w_2$ so that $w_2=x_jv$ has this property.

    Suppose $w_1=x_iu=\bdx^{\bda}$, $w_2=x_jv=\bdx^{\bdb}$ and $t\in[n]$ such that $\bda(k)=\bdb(k)$ for $k=1,2,\dots,t-1$ and $\bda(t)>\bdb(t)$.  We need to find suitable $l>t$ such that $x_t(w_2/x_l)\in I\wedge \frakm$. Notice that since $i\notin \supp(u)$, $\bda(i)=1$. Similarly, $\bdb(j)=1$.

    We have the following several cases.
    \begin{enumerate}[a]
        \item \label{part-a} When $j\le t<i$, $\deg_{x_k}(u)=\deg_{x_k}(v)$ for $1\le k <j$ and $\deg_{x_j}(u)=\bda(j)>\deg_{x_j}(v)=\bdb(j)-1=0$. As $u,v\in G(I)$ and $I$ is weakly polymatroidal, we can find $l>j$ such that $w=x_j(v/x_l)\in I$. Now $x_lw=x_jv=w_2$.

            If $w\in G(I)$, as $\suppdeg(w)=\suppdeg(v)$ and $j\notin\supp(v)$, we have $l\notin \supp(w)$. But since $l>j$, this contradicts the choice of $x_j$ and $v$.

            If $w\notin G(I)$, we can write $w=w'w''$ with $w'\in G(I)$ and $\deg(w'')\ge 1$. As $\suppdeg(w_2)=\suppdeg(v)+1>\suppdeg(w')$, we can find suitable $k\in \supp(x_lw'')\setminus\supp(w')$. Now $x_kw'$ divides $w_2$ and $x_kw'\ne w_2$. This contradicts the choice of $w_2\in G(I\wedge \frakm)$.

        \item \label{part-b} When $t<i$ and $j> t$, $\deg_{x_k}(u)=\deg_{x_k}(v)$ for $1\le k <t$ and $\deg_{x_t}(u)>\deg_{x_t}(v)$. There is some $l>t$ such that $w=x_t(v/x_l)\in I$.
            We can write $w=w'w''$ with $w'\in G(I)$.
            As $j\ne t$ and $j\notin \supp(v)$, we have  $j\notin \supp(w')$. Now $x_t(w_2/x_l)=(x_jw')w''\in I\wedge \frakm$.

        \item When $t\ge i=j$, as $\bda(i)=\bdb(j)=1$, we will have indeed $t>i=j$. Whence, $\deg_{x_k}(u)=\deg_{x_k}(v)$ for $1\le k<t$ and $\deg_{x_t}(u)>\deg_{x_t}(v)$. There is some $l>t$ such that $w=x_t(v/x_l)\in I$. As argued in part \ref{part-b}, we have $x_t(w_2/x_l)=x_jw\in I\wedge \frakm$.

        \item When $t\ge i>j$, we have $\deg_{x_k}(u)=\deg_{x_k}(v)$ for $1\le k<j$ and $\deg_{x_j}(u)=1>\deg_{x_j}(v)=0$. There exists some $l>j$ such that $w=x_j(v/x_l)\in I$. As in part \ref{part-a}, we get a contradiction.

        \item When $t>i$ and $j>i$, we have $\deg_{x_k}(u)=\deg_{x_k}(v)$ for $1\le k<i$ and $\deg_{x_i}(v)=1>\deg_{x_i}(u)=0$. There exists some $l>i$ such that $w=x_i(u/x_l)\in I$. As in part \ref{part-a}, we get a contradiction.
        \item When $t=i<j$, $x_t(w_2/x_j)=x_iv\in \frakm I$. As $\bda(i)=1>\bdb(i)$, $i\notin \supp(w_2)\supset \supp(v)$. Thus, $x_iv\in I\wedge \frakm$.
    \end{enumerate}
    And this completes the proof.
\end{proof}

\begin{example}
    In general, we cannot remove the support-degree assumption in Theorem \ref{WPP}. For instance, let $I=\braket{x_1x_2,x_2^3}\subset \KK[x_1,x_2,x_3]$. This is a weakly polymatroidal ideal. On the other hand, $I\wedge \frakm=\braket{x_1x_2^3,x_2^3x_3,x_1x_2x_3}$ is not weakly polymatroidal. 
\end{example}

As an immediate consequence of the Theorem \ref{WPP}, we have

\begin{corollary}
    [{cf.~\cite[Corollary 1.7]{MR2768496}}]
    \label{WPC}
    If ideal $I$ is a weakly polymatroidal ideal generated by monomials in one support-degree, then $I$ is support-componentwise weakly polymatroidal.
\end{corollary}

\begin{example}
    \label{NotCptWP}
    Not all weakly polymatroidal ideals are (support-)componentwise weakly polymatroidal. For instance, the ideal $I=\braket{x_1x_3,x_2x_3,x_1x_4x_5,x_2x_4x_5}$ in \cite[Example 1.8]{MR2768496} is weakly polymatroidal, but it is neither componentwise weakly polymatroidal nor support-componentwise weakly polymatroidal. 
\end{example}

Let $I$ be a monomial ideal. Following \cite{MR2557882}, we denote by $I_*$ the monomial ideal generated by the squarefree monomials in $I$ and call it the \Index{squarefree part} of $I$. Soleyman Jahan and Zheng \cite[Proposition 2.10]{MR2557882} showed that if $I$ has linear quotients, then $I_*$ has linear quotients.
We have a similar result for weakly polymatroidal ideals.

\begin{lemma}
    If $I$ is a weakly polymatroidal ideal, then $I_*$ is also weakly polymatroidal.
    \label{sqf-part}
\end{lemma}

\begin{proof}
    It suffices to mention that $G(I_*)=G(I)\cap I_*$. Now, an easy application of the definition completes the proof.
\end{proof}

\begin{proposition}
    \label{WP-sqf}
    If $I$ is a squarefree weakly polymatroidal ideal, then $I\wedge \frakm$ is again weakly polymatroidal.
\end{proposition}

\begin{proof}
    By \cite[Theorem 1.6]{MR2768496}, $I\frakm$ is weakly polymatroidal. Since $I$ is squarefree, $(I\frakm)_*=I\wedge \frakm$. Now, apply Lemma \ref{sqf-part}.
\end{proof}

\begin{definition}
    [{cf.~\cite[Theorem 2.5]{MR2845598}}]
    \begin{enumerate}[a]
        \item For two nonempty subsets $F,G$ of $[n]$, we say $F\succ_{lex} G$ if $\prod_{i\in F}x_i \succ_{lex} \prod_{i \in G}x_i$.
        \item Let $\Delta$ be simplicial complex with facets $\calF(\Delta)=\Set{F_1,\dots,F_s}$. $\Delta$ is called \Index{weakly co-polymatroidal} if for each pair of facets $F$ and $G$ with $F\succ_{lex} G$ and $i$ the smallest integer in $G\setminus F$, there exists some integer $j>i$ such that $j\notin G$ and $(G\setminus\Set{i})\cup \Set{j}\in \Delta$. If after an reorder of numbers in $[n]$, $\Delta$ becomes weakly co-polymatroidal, we will say that $\Delta$ is \Index{essential weakly co-polymatroidal}.
    \end{enumerate}
\end{definition}

It is observed in \cite[Theorem 2.5]{MR2845598} that $\Delta$ is weakly co-polymatroidal if and only if it is the Eagon complex of some weakly polymatroidal squarefree monomial ideal. Applying Proposition \ref{WP-sqf}, we get the following result.

\begin{corollary}
    All facet skeletons of weakly co-polymatroidal simplicial complexes are again weakly co-polymatroidal.
\end{corollary}

For pure simplicial complexes, its facets skeletons coincide with those $\Delta^{(r,s)}$. Thus, the above Corollary is equivalent to saying that all the skeletons $\Delta^{(r,s)}$ of pure weakly co-polymatroidal complexes are again weakly co-polymatroidal. The purity requirement here is crucial, as can be seen from Example \ref{NotCptWP}.

Suppose that $I$ is weakly polymatroidal and $G(I)=\Set{u_1\succ_{lex} u_2 \succ_{lex} \cdots \succ_{lex} u_m}$. Mohammadi and Moradi \cite[Theorem 1.3]{MR2768496} demonstrated that this ideal has linear quotients with respect to the given order.  Although general weakly polymatroidal ideals are not componentwise weakly polymatroidal, ideals with linear quotients have componentwise linear quotients by \cite[Theorem 2.7]{MR2557882}. Their proof relies on the fact (see also \cite[Proposition 2.9]{MR2557882}) that the admissible order of each component can be chosen to be compatible with the admissible order of the original ideal and the multiplication by the graded maximal ideal $\frakm$. We will investigate the weakly polymatroidal property with respect to this assumption.

\begin{lemma}
    \label{comp-lex}
    Let $u$ and $v$ be two monomials such that $\deg(u)< \deg(v)$.  Then $u\succ_{lex} v$ if and only if all $ux_i\succ_{lex} v$ for $i\in [n]$.
\end{lemma}

\begin{proof}
    Suppose that $u\succ_{lex} v$. By definition, there exists some $t\in [n]$ such that $\deg_{x_t}(u)>\deg_{x_t}(v)$ and $\deg_{x_k}(u)=\deg_{x_k}(v)$ for all $1\le k<t$. Now, take arbitrary $i\in[n]$. If $i\le t$, then $\deg_{x_i}(x_iu)=\deg_{x_i}(u)+1>\deg_{x_i}(v)$, and $\deg_{x_k}(x_iu)=\deg_{x_k}(u)=\deg_{x_k}(v)$ for all $1\le k<i$. Thus, $ux_i\succ_{lex}v$. If instead $i>t$, then $\deg_{x_t}(x_iu)=\deg_{x_t}(u)>\deg_{x_t}(v)$ and $\deg_{x_k}(x_iu)=\deg_{x_k}(v)$ for $1\le k<t$. Thus, again, $ux_i\succ_{lex}v$.

    Suppose that $ux_i\succ_{lex}v$ for all $i\in [n]$. In particular, $ux_n\succ_{lex}v$. Thus, there exists some $t\in[n]$ such that $\deg_{x_t}(ux_n)>\deg_{x_t}(v)$ and $\deg_{x_k}(ux_n)=\deg_{x_{k}}(v)$ for all $1\le k<t$. If this $t<n$, then $\deg_{x_t}(u)=\deg_{x_t}(ux_n)>\deg_{x_t}(v)$, and $\deg_{x_k}(u)=\deg_{x_k}(ux_n)=\deg_{x_k}(v)$ for all $1\le k <t$. Thus, $u\succ_{lex}(v)$. If instead $t=n$, then $\deg(u)+1=\deg(ux_n)>\deg(v)$, i.e., $\deg(u)\ge \deg(v)$. But this contradicts our assumption.
\end{proof}

\begin{corollary}
    \label{equiv-lex}
    Let $I$ be a monomial ideal in $S$. The following two conditions are equivalent:
    \begin{enumerate}[a]
        \item For each component $I_a$, elements of $G(\frakm I_{a-1})$ form the initial part of $G(I_{a})$ lexicographically.
        \item\label{equiv-lex-b} For each pair of monomials $u,v\in G(I)$, if $\deg(u)<\deg(v)$, then $u\succ_{lex} v$.
    \end{enumerate}
\end{corollary}

\begin{proposition}
    \label{eqv-cpwp}
    Consider the following three conditions for monomial ideal $I$.
    \begin{enumerate}[a]
        \item\label{cond1} $I$ is componentwise weakly polymatroidal.
        \item\label{cond2} For each component $I_a$, elements of $G(\frakm I_{a-1})$ form the initial part of $G(I_{a})$ lexicographically.
        \item\label{cond3} $I$ is weakly polymatroidal.
    \end{enumerate}
    Then the conditions \ref{cond1} and \ref{cond2} together will imply the condition \ref{cond3}.
\end{proposition}

\begin{proof}
    Let $u,v\in G(I)$ with $u\succ_{lex}v$. Say, $\deg(u)=a$ and $\deg(v)=b$. Then $u\in G(I_a)$ and $v\in G(I_b)$.  It follows from Corollary \ref{equiv-lex} that $a\le b$.  Take $u'=ux_n^{b-a}\succ_{lex} v$. If $j$ is the smallest index such that $\deg_{x_j}(u)>\deg_{x_j}(v)$, the same index works if we replace $u$ by $u'$. Notice that $u'\in G(I_b)$. As $I_{b}$ is weakly polymatroidal, there exists $l>j$ with $x_j(v/x_l)\in I_b \subset I$.  Thus, $I$ is weakly polymatroidal.
\end{proof}

\begin{example}
    Let $I=\braket{x_1^3,x_1^2x_2,x_1^2x_3,x_2^2x_3,x_2x_3^2,x_1x_3^3}\subset \KK[x_1,x_2,x_3]$. It is not difficult to check that $I$ is both weakly polymatroidal and componentwise weakly polymatroidal, but elements of $G(\frakm I_3)$ do not form the initial part of $G(I_4)$ lexicographically.  Thus, the condition \ref{cond2} in Proposition \ref{eqv-cpwp} is not a necessary condition for \ref{cond3}.
\end{example}

\begin{example}
    Let $I=\braket{x_2,x_3x_4}$. This ideal is weakly polymatroidal. It is easy to see that $G(\frakm I_{1})$ forms the initial part of $G(I_2)$ lexicographically. But $I_2$ is not weakly polymatroidal.
    Thus, the conditions \ref{cond1} and \ref{cond3} in Proposition \ref{eqv-cpwp} together will not necessarily imply the condition \ref{cond2}.
\end{example}

\begin{example}
    Let $I=\braket{x_1^2x_2,x_1x_2^2,x_3^2,x_2x_3,x_1x_3}$. Then $I$ is componentwise weakly polymatroidal, but not weakly polymatroidal.  Thus, condition \ref{cond1} in Proposition \ref{eqv-cpwp} is not sufficient for condition \ref{cond2}.
\end{example}

We have a support component version for Proposition \ref{eqv-cpwp}.

\begin{proposition}
    \label{scpwp}
    Let $I$ be a support-componentwise weakly polymatroidal ideal in $S$ satisfying the following two conditions:
    \begin{enumerate}[a]
        \item For each support component $I_{\braket{a}}$, elements of $G(\frakm \wedge I_{\braket{a-1}})$ are contained in $G(I_{\braket{a}})$.
        \item For each pair of monomials $u,v\in G(I)$, if $\suppdeg(u)<\suppdeg(v)$, then $u\succ_{lex} v$.
    \end{enumerate}
     Then $I$ is weakly polymatroidal.
\end{proposition}

\begin{proof}
    By Lemma \ref{comp-lex}, it is clear that  elements of $G(\frakm \wedge I_{\braket{a-1}})$ form the initial part of $G(I_{\braket{a}})$ lexicographically.
    Let $u,v\in G(I)$ with $u\succ_{lex}v$. Say, $\suppdeg(u)=a$ and $\suppdeg(v)=b$. Then $a\le b$, $u\in G(I_{\braket{a}})$ and $v\in G(I_{\braket{b}})$.   Let $j$ be the smallest index such that $\deg_{x_j}(u)>\deg_{x_j}(v)$.

    If $a=b$, since $I_{\braket{a}}$ is weakly polymatroidal, there exists some $l> j$ such that $x_j(v/x_l)\in I_{\braket{a}}\subset I$.

    If $a< b$, we can find suitable monomial $u''\in \supp(v)\setminus \supp(u)$ such that $\suppdeg(u'')=\deg(u'')=b-a$ and $x_j\succ_{lex}u''$.  Take $u'=uu''$. Then $u'\in G(I_{\braket{b}})$ and $u'\succ_{lex} v$ such that $j$ is again the smallest index with $\deg_{x_j}(u')>\deg_{x_j}(v)$.  As $I_{\braket{b}}$ is weakly polymatroidal, there exists some integer $l>j$ with $x_j(v/x_l)\in I_{\braket{b}} \subset I$.

    Thus, $I$ is weakly polymatroidal.
\end{proof}

\subsection{Vertex decomposable complexes}

Bj\"orner and Wachs {\cite[Definition 11.1]{MR1401765} introduced the notion of vertex decomposability for nonpure simplicial complexes as follows.

\begin{definition}
    \label{VertexDecomposable}
    A simplicial complex $\Delta$ is \Index{vertex decomposable} if it is a simplex or $\Set{\emptyset}$, or there exists a vertex $v$, called \Index{shedding vertex}, such that
    \begin{enumerate}[a]
        \item $\Delta\setminus v$ and $\link_{\Delta}(v)$ are vertex decomposable, and
        \item \label{ptb-df} no facet of $\link_{\Delta}(v)$ is a facet of $\Delta\setminus v$.
    \end{enumerate}
\end{definition}

This definition generalized the original one by Provan and Billera \cite{MR0593648} for the pure complexes.

As for Question \ref{Skeleton}(c), Woodroofe \cite[Lemma 3.10]{MR2853065} showed that all skeletons $\Delta^{(0,s)}$ of vertex decomposable simplicial complexes are again vertex decomposable. This result was also proved for pure vertex decomposable simplicial complex by Swanson \cite[Corollary 2.86]{MR2941472}.

The diagram in Figure \ref{dia-cpx} displays some properties of simplicial complexes and their relationships.  In this diagram, weakly co-stability implies vertex decomposability by\cite[Theorem 16]{MR1704158} under some additional mild condition; weakly co-polymatroidal property also implies vertex decomposability by \cite[Theorem 2.5]{MR2845598}; vertex decomposability implies shellability by \cite[Theorem 11.3]{MR1401765}. The remaining two implications follow from them.

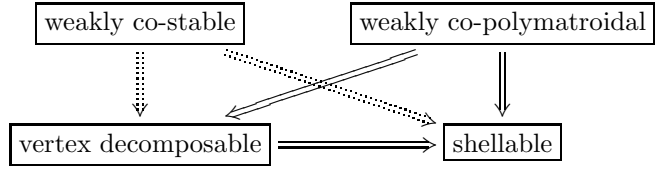
\begin{figure}
    \centerline{
        \xymatrix{
            \framebox{weakly co-stable} \ar@{:>}[rd]\ar@{:>}[d] & \framebox{weakly co-polymatroidal} \ar@{=>}[ld] \ar@{=>}[d] \\
            \framebox{vertex decomposable} \ar@{=>}[r] & \framebox{shellable}
        }
    }
    \caption{Relations among classes of simplicial complexes}
    \label{dia-cpx}
\end{figure}

On the other hand, we have the diagram in Figure \ref{dia-ideal}, displaying some properties of monomial ideals and their relationships. In this diagram, weakly $I$-stable ideals have linear quotients by Proposition \ref{WIS-LQ}; weakly polymatroidal ideals have linear quotients by \cite[Theorem 1.3]{MR2768496}.

\begin{figure}
    \centerline{
        \xymatrix{
            \framebox{weakly $I$-stable} \ar@{=>}[rd] \ar@{-->}[d] & \framebox{weakly polymatroidal} \ar@{=>}[d]\ar@{-->}[ld] \\
            \framebox[1cm]{?} \ar@{-->}[r] & \framebox{linear quotients}
        }
    }
    \caption{Relations among classes of monomial ideals}
    \label{dia-ideal}
\end{figure}
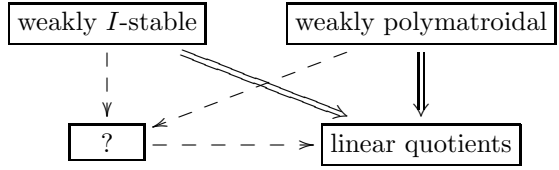

Therefore, we are interested in the missing class of monomial ideals in Figure \ref{dia-ideal}, which corresponds to the vertex decomposable complexes in Figure \ref{dia-cpx}. We name it as \Index{variable decomposable ideals}.

To start with, we look at its squarefree subclass as follows.
When $\Delta$ is vertex decomposable, we say $I_{\Delta^\vee}$ is \Index{variable decomposable}. Considering Observation \ref{dual-observation}, we know the class of variable decomposable squarefree monomial ideals can be defined recursively as follows.
\begin{enumerate}[a]
    \item Ideals generated by variables are variable decomposable.
    \item If there is a variable $x_v$, such that the ideals
        \[
            I_0:=\Braket{u: u\in G(I)\text{ and } x_v \text{ does not divide }u}
        \]
        and
        \[
            I_1:=\Braket{u/x_v : u\in G(I) \text{ and }x_v\text{ divides }u}
        \]
        are variable decomposable and $I_0\subset I_1 \frakm$, then $I$ is also variable decomposable. This variable $x_v$ shall be called a \Index{shedding variable} in this case.
\end{enumerate}

Notice that $I_{(\Delta\setminus v)^\vee}=\braket{x_v}\cap I_{\Delta^\vee}=x_v(I_1+I_0)$ and $I_{(\link_{\Delta}(v))^\vee}=x_vI_0$. The condition \ref{ptb-df} in Definition \ref{VertexDecomposable} is translated to the condition that $G(I_1+I_0)\cap G(I_0)=\emptyset$. As these ideals are squarefree, it is not difficult to see that the following conditions are equivalent:
\begin{enumerate}[a]
    \item $G(I_1+I_0)\cap G(I_0)=\emptyset$.
    \item $I_0\subset I_1\frakm$.
    \item $I_0\subset I_1\wedge \frakm$.
\end{enumerate}

We can generalize this approach to monomial ideals which are not necessarily squarefree.

\begin{definition}
    \label{VD}
    A monomial ideal $I\subset S=\KK[x_1,\dots,x_n]$ is \Index{variable decomposable} (resp.~strongly variable decomposable) if
    \begin{enumerate}[a]
        \item $I=\braket{0}$ or $\braket{1}$, or
        \item there is a variable $x_v$ with $r=\max\Set{\deg_{x_v}(u): u\in G(I)}$, such that all the ideals
            \[
                I_i:=\Braket{u/x_v^i: u\in G(I)\text{ and }\deg_{x_v}(u)=i}\subset \KK[x_1,\dots,\widehat{x_v},\dots,x_n], \quad i=0,1,\dots, r
            \]
            are variable decomposable (resp.~strongly variable decomposable), and for each $i=1,2,\dots,r$,
            $I_{i-1}\subset I_{i} \frakm$ (resp.~$I_{i-1}\subset I_i\wedge \frakm$).
            In this case, $x_v$ shall be called a \Index{shedding variable} as well.
    \end{enumerate}
\end{definition}

Thus, ideals generated by variables are strongly variable decomposable. Notice that
we have the implication
\[
    \text{strongly variable decomposable} \implies \text{variable decomposable}.
\]
And a squarefree monomial ideals is variable decomposable if and only if it is strongly variable decomposable.

\begin{theorem}
    Let $I\subset S$ be a monomial ideal and $\frakm$ the graded maximal ideal of $S$.
    \begin{enumerate}[a]
        \item If $I$ is variable decomposable, then $I\frakm$ is again variable decomposable.
        \item If $I$ is strongly variable decomposable, then $I\wedge \frakm$ is again strongly variable decomposable.
    \end{enumerate}
\end{theorem}

\begin{proof}
    We only prove the variable decomposable case. The strongly variable decomposable case is similar.

    Suppose that $I$ is variable decomposable. Without loss of generality, we assume that $I$ is neither $\braket{0}$ nor $\braket{1}$, and $x_n$ is a shedding variable for $I$.  Write $S'=\KK[x_1,\dots,x_{n-1}]$ and $\frakm'=\braket{x_1,\dots,x_{n-1}}$ its graded maximal ideal.  Notice that each $I_k$ in Definition \ref{VD} satisfies $I_k=(I_k\cap S')S$. Thus, the condition $I_{k-1}\subset I_k\frakm$ is equivalent to saying $I_{k-1}\subset I_k\frakm'$. Now
    \begin{align*}
        I\frakm=& \left(\sum_{k=0}^r x_n^k I_k\right)(\braket{x_n}+\frakm') \\
        = & \sum_{k=0}^r x_n^kI_k x_n +\sum_{k=0}^r x_n^k (I_k\frakm') \\
        = & \sum_{k=1}^{r+1} x_n^k I_{k-1} +\sum_{k=0}^r x_n^k (I_k\frakm') \\
        =& \left(\sum_{k=0}^r x_n^k (I_k\frakm')\right) + x_n^{r+1}I_r.
    \end{align*}
    For each $k=1,\dots,r$, we have $I_{k-1}\frakm' \subset I_k\frakm' \frakm$. For the $k=r+1$ case, we also have $I_r\frakm' \subset I_r \frakm$. Thus, an induction argument implies that $I\frakm$ is variable decomposable and $x_n$ is again a shedding variable.
\end{proof}

\begin{corollary}
    If $\Delta$ is a vertex decomposable simplicial complex, then its facet skeletons are again vertex decomposable.
\end{corollary}

In the remaining part of this paper, we show that the class of variable decomposable ideals fills the gap in Figure \ref{dia-ideal}. We have observed that Eagon complexes of variable decomposable squarefree monomial ideals are precisely those vertex decomposable complexes.

The following result generalizes the fact that vertex decomposable simplicial complexes are shellable.

\begin{proposition}
    \label{VDLQ}
    Variable decomposable monomial ideals have linear quotients.
\end{proposition}

\begin{proof}
    Suppose that $I\subset S=\KK[x_1,\dots,x_n]$ is weakly variable decomposable and $x_n$ is a shedding variable. By induction, we may assume that in the representation $I=\sum_{k=0}^r x_n^k I_k$ as in Definition \ref{VD}, each $I_k$ has linear quotients with respect to the order $u_{k,1},\dots,u_{k,t_k}$. We claim that
    \[
        x_n^r u_{r,1},\dots,x_n^r u_{r,t_r},x_{n}^{r-1}u_{r-1,1},\dots,x_n^{r-1}u_{r-1,t_{r-1}},\dots,u_{0,1},\dots,u_{0,t_0}
    \]
    is an admissible order for $I$ to have linear quotients.

    To verify this claim, it suffices to check the linear quotient condition for the monomial $u_{0,j}$.  Write
    \[
        I'=\braket{x_n^r u_{r,1},\dots,x_n^r u_{r,t_r},x_{n}^{r-1}u_{r-1,1},\dots,x_n^{r-1}u_{r-1,t_{r-1}},\dots,x_nu_{1,1},\dots,u_{1,t_1}}
    \]
    and
    \[
        I''=\braket{u_{0,1},\dots,u_{0,j-1}}.
    \]
    
    We first look at $I':u_{0,j}$.
     As $I'\subset \braket{x_n}$ and $u_{0,j}\notin \braket{x_n}$, we have $I':u_{0,j} \subset \braket{x_n}$. On the other hand, as $u_{0,j}\in I_0\subset I_1\frakm$, $u_{0,j}=u_{1,k}u'$ for some $k\le t_1$ and some monomial $u'\in S$. Now, $x_nu_{0,j}=(x_nu_{1,k})u'\in I'$, i.e., $x_n\in I':u_{0,j}$. This implies that $I':u_{0,j}=\braket{x_n}$. 
    
    Meanwhile, $I'':u_{0,j}$ is generated by variables by the assumption for $I_0$. Thus, $(I'+I''):u_{0,j}$ is generated by variables.
\end{proof}

The following result generalizes the fact in \cite[Theorem 2.5]{MR2845598} that weakly co-polymatroidal complexes are vertex decomposable.

\begin{proposition}
    \label{WPVD}
    Weakly polymatroidal ideals are variable decomposable.
\end{proposition}

\begin{proof}
    Let $I$ be a weakly polymatroidal ideal in $S=\KK[x_1,\dots,x_n]$.  We will show that $x_1$ is a shedding variable for $I$ to be variable decomposable.
Let $x_v=x_1$ in Definition \ref{VD} and write $I=\sum_{k=0}^r x_1^k I_k$.  We may assume that $r>0$ and $I_0$ is nonzero.

    Take arbitrary monomials $u'\in G(I_k)$ with $0\le k \le r-1$ and $u''\in G(I_{r})$.  As $x_1^k u'$ and $x_1^{r}u''$ belong to $G(I)$ and $x_1^k u'\prec_{lex}x_1^{r}u''$, we can find suitable $j>1$ such that $f=x_1(x_1^ku'/x_j)\in I$. Since $\deg_{x_1}(f)=k+1$, this $f\in \sum_{i=0}^{k+1}x_1^i I_i$. If $f\in x_1^iI_i$ with $i\le k$, $f/x_1^{k+1}=u'/x_j\in I_i$, contradicting to the fact that $x_1^ku' \in G(I)$. Thus, indeed, $f\in x_1^{k+1}I_{k+1}$. Now, $u'/x_j\in I_{k+1}$ and therefore $u' \in \frakm I_{k+1}$. By the arbitrariness of $u'$, we obtain that
    \[
        I_k\subset \frakm I_{k+1}.
    \]
    From this relation and the assumption that $I_0\ne 0$, we also know that all $I_k\ne 0$ for $k=0,1,\dots,r$.

    We claim that each $I_k$ is weakly polymatroidal. Take two $u'$ and $u''$ in $G(I_k)$. Suppose that $u'\prec_{lex}u''$, and $t$ is the smallest index such that $\deg_{x_t}(u')< \deg_{x_t}(u'')$. Then
    $x_1^ku'\prec_{lex}x_1^ku''$, and $t$ is the smallest index such that $\deg_{x_t}(x_1^ku')< \deg_{x_t}(x_1^ku'')$.  As $I$ is weakly polymatroidal and $x_1^ku',x_1^ku''\in G(I)$, we have suitable $j>t$ such that $g=x_t(x_1^ku'/x_j)\in I$. Since $\deg_{x_1}(g)=k$, we have $g\in \sum_{i=0}^k x_1^i I_i$.  Suppose that $g\in x_1^i I_i$ for some $i\le k$. Then $x_tu'/x_j\in I_i\subset \frakm^{k-i}I_k\subset I_k$. Therefore, $I_k$ is weakly polymatroidal. Now, by induction, $I_k$ is variable decomposable.
\end{proof}

In the following, we discuss weakly $I$-stable ideals. Here, $I$ is a monomial ideal generated by monomials of the form $x_i^{a_i}$. For convenience, we set $a_i=+\infty$ if $x_i^d \notin I$ for all $d\in \NN$.

The following result generalizes the fact in \cite[Theorem 16]{MR1704158} that every weakly co-stable complexes are vertex decomposable. Notice that in their proof, it is implicitly assumed that \emph{when $F$ is a facet of $\link_{\Delta^*}(1)$ and $F\ne [n]\setminus \Set{1}$, the cardinality of the set $[n]\setminus F$ is at least $2$}. When this assumption is not satisfied, the complex $\braket{\Set{3},\Set{1,2}}$, which corresponds to the squarefree weakly stable ideal $J=\braket{x_1x_2,x_3}$, provides a counterexample to \cite[Theorem 16]{MR1704158}. 

We rephrase the above assumption as follows.

\begin{definition}
Let $J$ be a monomial ideal in $S=\KK[x_1,\dots,x_n]$. For each nonempty subset $F\subset [n]$ and vector $\bda\in \NN^n$ with $\supp(\bda)\subset F$, let $t=\max(F)$ and
\[
J_{F,\bda}=\Braket{u/\bdx^{\bda}: u\in G(J) \text{ and }\deg_{x_j}(u)=\bda(j) \text{ for all }j\le t}.
\]
We say $J$ is \Index{sequentially pure} if for each $F$ and $\bda$ above, whenever $x_j\in G(J_{F,\bda})$ for some $j$ (which is necessarily greater than $t$), then
\begin{enumerate}[1]
    \item $J_{F,\bda}$ is linear, and
    \item either $J_{F,\bda+\bde_t}=\braket{1}$ or $J_{F,\bda+k\bde_t}=\braket{0}$ for all integer $k\ge 1$.
\end{enumerate}
Here $\bde_t$ in the $t^{\th}$ standard basis vector.
\end{definition}

Obviously, ideals generated by monomials in one degree are sequentially pure.

\begin{proposition}
    Let $J$ be a weakly $I$-stable ideal in $S=\KK[x_1,\dots,x_n]$.  If $J$ is sequentially pure, then $J$ is variable decomposable.
\end{proposition}

\begin{proof}
   We will show that $x_1$ is a shedding variable for $J$ to be variable decomposable. We also write $S'=\KK[x_2,x_3,\dots,x_n]$.
    Let $x_v=x_1$ in Definition \ref{VD} and write $J=\sum_{k=0}^r x_1^k J_k$ correspondingly.  We may assume that $r>0$ and $J_0$ is nonzero. Obviously $r<a_1$.

    Take arbitrary monomial $u\in G(J_k)$ with $0\le k \le r-1$. As $x_1^k u\in G(I)$ with $k<r$, this $u\ne 1$.  If $\deg(u)=1$, then $k=r-1$, $J_{r-1}$ is linear and $J_r=\braket{1}$. There is not much to show in this case.  Therefore, we may assume that $\deg(u)>1$. Now, for each $j<\max(u')=\max( (x_1^ku)')$, we have some $i>j$ such that $x_j(x_1^ku/x_i)\in J+I$.

    When $j=1$, this means that $x_1^{k+1}u/x_i\in J+I$. If $x_1^{k+1}u/x_i\in I$, this simply implies that $k+1=a_1$. But as $k<r<a_1$, this is impossible. Thus, $x_1^{k+1}u/x_i\in J$. By an argument similar to the proof for Proposition \ref{WPVD}, we can again deduce that $u/x_i\in J_{k+1}$. Therefore,
    \[
        J_k\subset \frakm J_{k+1}.
    \]
    From this relation and the assumption that $J_0\ne 0$, we also know that all $J_k\ne 0$ for $k=0,1,\dots,r$.

    When $j>1$, a similar argument implies that $x_ju/x_i\in I+J_k$. This implies that $J_k\cap S'$ is weakly $(I\cap S')$-stable.  We also observe that $J_k$ is sequentially pure.  Therefore, by induction, $J_k$ is variable decomposable.
\end{proof}


\end{document}